\documentclass{amsart}
\usepackage[latin1]{inputenc}
\usepackage[english]{babel}
\usepackage{indentfirst}
\usepackage{amssymb}
\usepackage{amsthm}
\usepackage{xcolor}
\usepackage[all]{xy}
\newtheorem{THEO}{Theorem}[section]
\newtheorem*{THEO*}{Theorem}

\newtheorem{CORO}[THEO]{Corollary}
\newtheorem{LEMM}[THEO]{Lemma}

\newtheorem{DEFI}[THEO]{Definition}

\newtheorem{REMA}[THEO]{Remark}
\newcommand{\sub}{\subseteq}
\def\({\left(}
\def\){\right)}

\def\N{\mathbb{ N}}
\def\R{\mathbb{ R}}

\def\rto{\rightarrow}

\def\CC{\mathcal{C}}

\def\PP{\mathcal{P}}

\def\1{\textbf{1}}

\def\supp{\operatorname{supp}}

\def\dens{\operatorname{dens}}

\newcommand{\Suc}[2]{\ensuremath{\left({#1}\right)_{{#2}=1}^\infty}}

\def\w{\omega}

\def\Int{\operatorname{Int}}

\usepackage{mathtools}
\DeclarePairedDelimiterX{\norm}[1]{\lVert}{\rVert}{#1}
\def\nonSN{{\rm {non}}({\mathcal {SN}})}
\def\covM{{\rm {cov}}({\mathcal M})}

\author{Gonzalo Mart\'{i}nez-Cervantes}
\address{Departamento de Matem\'{a}ticas\\ Facultad de Matem\'{a}ticas\\ Universidad de Murcia\\ 30100 Espinardo (Murcia)\\ Spain} 
\email{gonzalo.martinez2@um.es}
\thanks{This work was supported by the research project 19275/PI/14 funded by Fundaci\'{o}n S\'{e}neca - Agencia de Ciencia y Tecnolog\'{i}a de la Regi\'{o}n de Murcia within the framework of PCTIRM 2011-2014.  This work was also supported by Ministerio de Econom\'{i}a y Competitividad and FEDER (project MTM2014-54182-P)}
\title{Riemann integrability versus weak continuity}
\begin{document}
\begin{abstract}
In this paper we focus on the relation between Riemann integrability and weak continuity.
A Banach space $X$ is said to have the weak Lebesgue property if every Riemann integrable function from $[0,1]$ into $X$ is weakly continuous almost everywhere. We prove that the weak Lebesgue property is stable under $\ell_1$-sums and obtain new examples of Banach spaces with and without this property. Furthermore, we characterize Dunford-Pettis operators in terms of Riemann integrability and provide a quantitative result about the size of the set of $\tau$-continuous non Riemann integrable functions, with $\tau$ a locally convex topology weaker than the norm topology.
\end{abstract}
\maketitle
\section{Introduction}
The study of the relation between Riemann integrability and continuity on Banach spaces started on 1927, when Graves showed in \cite{MR1501382} the existence of a vector-valued Riemann integrable function not continuous almost everywhere (a.e.~for short). Thus, the following problem arises:
\begin{enumerate}
\item[] \textit{Given a Banach space $X$, determine necessary and sufficient conditions for the Riemann integrability of a function $f:[0,1] \rto X$.}
\end{enumerate} 

A Banach space $X$ for which every Riemann integrable function $f \colon [0,1] \rto X$ is continuous a.e.~is said to have the Lebesgue property (LP for short).
All classical infinite-dimensional Banach spaces except $\ell_1$ do not have the LP.
For more details on this topic, we refer the reader to \cite{gordon1991}, \cite{Filho79}, \cite{Pizzotti89}, \cite{zbMATH03983856} and \cite{Naralenkov2007}.

Regarding weak continuity, Alexiewicz and Orlicz constructed in 1951 a Riemann integrable function which is not weakly continuous a.e.~\cite{MR0043366}. 
A Banach space $X$ is said to have the weak Lebesgue property (WLP for short) if every Riemann integrable function $f:[0,1] \rto X$ is weakly continuous a.e.
This property was introduced in \cite{Wang1996}. Every Banach space with separable dual has the WLP and the example of \cite{MR0043366} shows that $\CC([0,1])$ does not have the WLP. Other spaces with the WLP, such as $L^1([0,1])$, can be found in \cite{MR2591586} and \cite{Wang2001}. In this paper we focus on the relation between Riemann integrability and weak continuity. In Section 2 we present new results on the WLP. In particular, we prove that the James tree space $JT$ does not have the WLP (Theorem \ref{JTWLP}) and we study when $\ell_p(\Gamma)$ and $c_0(\Gamma)$ have the WLP in the nonseparable case (Theorem \ref{c0lpWLP}). Moreover, we prove that the WLP is stable under $\ell_1$-sums (Theorem \ref{l1sumaWLP}) and we apply this result to obtain that $\CC(K)^\ast$ has the WLP for every compact space $K$ in the class $MS$ (Corollary \ref{COROMS}).

Alexiewicz and Orlicz also provided in their paper an example of a weakly continuous non Riemann integrable function. 
V. Kadets proved in \cite{Kadets1994} that a Banach space $X$ has the Schur property if and only if every weakly continuous function $f:[0,1] \rto X$ is Riemann integrable. Wang and Yang extended this result in \cite{Wang2000} to arbitrary locally convex topologies weaker than the norm topology. In the last section of this paper we give an operator theoretic form of these results that, in particular, provides a positive answer to a question posed by Sofi in \cite{pre06123790}.

\subsection*{Terminology and Preliminaries}

All Banach spaces are assumed to be real. In what follows, $X^\ast$ denotes the dual of a Banach space $X$. The weak and weak$^\ast$ topologies of $X$ and $X^\ast$ will be denoted by $\w$ and $\w ^\ast$ respectively.
By an operator we mean a linear continuous mapping between Banach spaces. The Lebesgue measure in $\R$ is denoted by $\mu$. The interior of an interval $I$ will be denoted by $Int(I)$.
The density character $\dens(T)$ of a topological space $T$ is the minimal cardinality of a dense subset.

A partition of the interval $[a,b] \subset \R$ is a finite collection of non-overlapping closed subintervals covering $[a,b]$. A tagged partition of the interval $[a,b]$ is a partition $\lbrace [t_{i-1}, t_i]: 1 \leq i \leq N \rbrace $ of $[a,b]$ together with a set of points $\lbrace s_i: 1 \leq i \leq N \rbrace $ that satisfy $s_i \in (t_{i-1}, t_i)$ for each $i$.
Let $\mathcal{P}=\lbrace (s_i, [t_{i-1}, t_i]): 1\leq i \leq N \rbrace $ be a tagged partition of $[a,b]$.
For every function $f:[a,b] \rto X$, we denote by $f( \mathcal{P} )$ the Riemann sum $\sum \limits_{i=1}^N (t_i - t_{i-1})f(s_i)  $. The norm of $\mathcal{P}$ is $\norm{ \mathcal{P} } := \max \lbrace t_i - t_{i-1}: 1 \leq i \leq N \rbrace$. We say that a function $f:[a,b] \rto X$ is Riemann integrable, with integral $x \in X$, if for every $\varepsilon >0$ there is $\delta >0$ such that $\| f ( \mathcal{P} ) - x \| < \varepsilon $ for all tagged partitions $\mathcal{P}$ of $[a,b]$ with norm less than $\delta$. In this case, we write $x= \int_a^b f(t)dt$.

The following criterion will be used for proving the existence of the Riemann integral of certain functions:

\begin{THEO}[\cite{gordon1991}]

Let $f:[0,1] \rto X$. The following statements are equivalent:
\begin{enumerate}

\item The function $f$ is Riemann integrable.
\item For each $\varepsilon >0 $ there exists a partition $\mathcal{P}_\varepsilon$ of $[0,1]$ with $\|f( \mathcal{P}_1) - f(\mathcal{P}_2) \| < \varepsilon $ for all tagged partitions $\mathcal{P}_1$ and $\mathcal{P}_2$ of $[0,1]$ that have the same intervals as $\mathcal{P}_\varepsilon$.
\item There is $x \in X$ such that for every $\varepsilon >0$ there exists a partition $\mathcal{P}_\varepsilon$ of $[0,1]$ such that $\| f( \mathcal{P} )- x \| < \varepsilon $ whenever $\mathcal{P}$ is a tagged partition of $[0,1]$ with the same intervals as $\mathcal{P}_\varepsilon$.   
\end{enumerate}
\end{THEO}

We will also be concerned about cardinality. Throughout this paper, $\mathfrak{c}$ denotes the cardinality of the continuum and $\covM$ denotes the smallest cardinal such that there exist $\covM$ nowhere dense sets in $[0,1]$ whose  union is the interval $[0,1]$.
This cardinal coincides with the smallest cardinal such that there exist $\covM$ closed sets in $[0,1]$ with Lebesgue measure zero whose union does not have Lebesgue measure zero (see \cite[Theorem 2.6.14]{Bartoszynski199293}).

A set $A \subset \R$ is said to be \emph{strongly null} if for every sequence of positive reals $\Suc{\varepsilon_n}{n}$ there exists a sequence of intervals $\Suc{I_n}{n}$ such that $\mu (I_n) < \varepsilon_n$ for every $n \in \N$ and $A \subset \bigcup_{n\in \N} I_n$. We will be interested in the following result:

\begin{THEO}[\cite{Pawlikowski}]	
\label{Pawlikowski}
A set $A \subset \R$ is strongly null if and only if for every closed set $F$ with Lebesgue measure zero, the set 
$A+F = \lbrace a+z \colon a\in A \mbox{ and } z\in F \rbrace $ has Lebesgue measure zero.

\end{THEO}

We will denote by $\nonSN$ the smallest cardinal of a non strongly null set.
We have that $\aleph_1 \leq \covM \leq \nonSN \leq \mathfrak{c}$ and, 
under Martin's axiom, and therefore under the Continuum Hypothesis, $\nonSN = \covM = \mathfrak{c}$. 
Furthermore, if $\mathfrak{b}=\mathfrak{c}$ then $\nonSN = \covM$.
However, there exist models of ZFC satisfying $\covM < \nonSN$.  
For further references and results on this subject we refer the reader to \cite{bartoszynski1995set}.

\section{The weak Lebesgue property}

It is known that every Banach space with separable dual has the WLP \cite{Wang2001}. Next theorem gives a generalization in terms of $\covM$.

\begin{THEO}
\label{dualWLP}
Every Banach space $X$ such that $\dens(X^\ast) < \covM$ has the WLP. 
\end{THEO}
\begin{proof}
Let $D=\lbrace x_i^\ast \rbrace _{i \in I}$ be a dense subset in $X^\ast$ with $|I| < \covM$ and take 
$f:[0,1] \rto X$ a Riemann integrable function. Then every function $x_i^\ast f$ is Riemann integrable. Let $E_i$ be the set of points of discontinuity of $x_i^\ast f$ for every $i\in I$. Each $E_i$ is a countable union of closed sets with measure zero, so $E= \bigcup_{i \in I} E_i $ has measure zero since $|I|< \covM$. 
We claim that $f$ is weakly continuous at every point of $E^c$. Let $x^\ast \in X^\ast$ and let $M$ be an upper bound for $\lbrace \| f(t) \| : t \in [0,1] \rbrace$. Fix $\varepsilon >0$ and $t \in E^c$. Then, there exists $x_i ^\ast \in D $ such that $\| x_i ^\ast - x^\ast \| < \frac{ \varepsilon }{3M}$.  
Since $t \notin E_i$, there exists a neighbourhood $U$ of $t$ such that $| x_i^\ast f (t) - x_i ^\ast f(t')| < \frac{\varepsilon}{3}$ for every $t' \in U$. Thus, 
$$| x^\ast f (t) - x^\ast f(t') | \leq | x^\ast f (t) - x_i^\ast f (t) | + | x_i^\ast f (t) - x_i^\ast f (t')| + | x_i^\ast f (t')- x^\ast f(t') | < \varepsilon $$
for every $t' \in U$.
\end{proof}

\begin{CORO}

Every Banach space with separable dual has the WLP.

\end{CORO}

The space $\ell_1$ has the WLP because it has the LP. Since every asymptotic $\ell_1$ space has the LP \cite{Naralenkov2007}, the space $\Lambda_T$ defined by Odell in \cite{Odell1985} is a separable Banach space with nonseparable dual such that it does not contain an isomorphic copy of $\ell_1$ but it has the WLP (it is asymptotic $\ell_1$). On the other hand, the James tree space $JT$ (see \cite[Section 13.4]{opac-b1119539}) is a separable Banach space with nonseparable dual such that it does not contain an isomorphic copy of $\ell_1$ and  it does not have the WLP:

\begin{THEO}
\label{JTWLP}
The James tree space does not have the WLP.
\end{THEO}

\begin{proof}

We represent the dyadic tree by $$T=\lbrace (n,k): n=0,1,2,\dots \mbox{ and } k=1,2,\dots,2^ {n}\rbrace .$$ A node $(n,k) \in T$ has two immediate successors $(n+1, 2k-1)$ and $(n+1,2k)$. Then, a segment of $T$ is a finite sequence $\lbrace p_1,\dots,p_m \rbrace$ such that $p_{j+1}$ is an immediate successor of $p_j$ for every $j=1,2,\dots,m-1$. The James tree space $JT$ is the completion of $c_{00}(T)$ with the norm
$$\| x \| =\sup \sqrt{ \sum_{j=1}^l \( \sum_{(n,k)\in S_j} x(n,k) \)^2 } < \infty ,$$
where the supremum is taken over all $l\in \N$ and all sets of pairwise disjoint segments $S_1, S_2,\dots, S_l$.
Let $\lbrace e_{(n,k)} \rbrace _{(n,k) \in T}$ be the canonical basis of $JT$, i.e. $e_{(n,k)}$ is the characteristic function of $(n,k) \in T$.
Define $f:[0,1] 	\rto JT$ as follows:
$$ f (t) = \begin{cases}
\ e_{(n-1,k)} & \mbox{ if } t = \frac{ 2k-1}{2^n} \mbox{ with } n\in \N \mbox{ and } k=1,2,\dots,2^{n-1} \\
\ 0  & \mbox{ in any other case.}
\end{cases}
$$
We claim that $f$ is Riemann integrable. 
Fix $N \in \N$ and let $\lbrace I_1, I_2,\dots, I_{2^N -1} \rbrace $ be a family of closed disjoint intervals of  $[0,1]$ with $$ \sum_{1\leq n \leq 2^N-1} \mu (I_n) \leq \frac{1}{2^N} \mbox{ and } \frac{n}{2^N} \in Int(I_n)  \mbox{ for each } 1 \leq n \leq 2^{N}-1.$$
Let $ J_1, J_2, \dots , J_{2^N}$ be the closed disjoint intervals of $[0,1]$ determined by $$[0,1] \setminus \bigcup \limits_{1 \leq n \leq 2^N-1} \Int(I_n).$$
Then, $\mu(J_n) \leq \frac{1}{2^N}$ and $\| \sum_{n=1}^{2^N} a_n f(t_n) \| \leq \sqrt{\sum_{n=1}^{2^N} a_n ^2 } $ for every $a_n \in \R$ and every $t_n \in J_{n} $ due to the definition of the norm in $JT$.
Thus, any tagged partition $\mathcal{P}_N$ with intervals $ J_1, I_1, J_2, \dots, I_{2^N -1}, J_{2^N}$ and points $t_1, t_1 ', t_2, \dots,t_{2^N-1}', t_{2^N}$ satisfies
$$\| f( \mathcal{P}_N ) \| \leq \norm[\bigg]{ \sum_{n=1}^{2^N} \mu (J_n) f(t_{2n-1}) } +   \sum_{n=1}^{2^N-1}  \mu (I_n) \leq $$
$$\leq \sqrt{\sum_{n=1}^{2^N} \mu (J_n)^2 } + \frac{1}{2^N} \leq \sqrt{\sum_{n=1}^{2^{N}} \frac{1}{2^{2N}} } + \frac{1}{2^N} \leq  \frac{2}{\sqrt{2^N}}.$$
Hence, $\| f( \mathcal{P}_N ) \| \xrightarrow{N \rto \infty } 0$ and $f$ is Riemann integrable with integral zero.

We show that $f$ is not weakly continuous at any irrational point $t \in [0,1]$.  
Fix a irrational point $t\in [0,1]$. There exists a sequence of dyadic points $\Suc{ \frac{ 2k_j-1}{2^{n_j}} }{j}$ converging to $t$ with $\Suc{n_j-1, k_j}{j} $ a sequence in $T$ such that $(n_{j+1}-1, k_{j+1})$ is an immediate successor of $(n_j-1, k_j)$ for every $j\in \N$.
Then, $ \sum_{j=1}^\infty e_{(n_j-1, k_j)}^\ast$ is a functional in $ JT^\ast$, so the sequence $f(\frac{ 2k_j-1}{2^{n_j}}) = e_{(n_j-1, k_j )}$ is not weakly null and $f$ is not weakly continuous at $t$.  
\end{proof}

\begin{CORO}[\cite{MR0043366}]

\label{C[0,1]noWLP}
$\CC([0,1]) $ does not have the WLP.

\end{CORO}

\begin{proof}

Since every subspace of a Banach space with the WLP has the WLP and every separable Banach space is isometrically isomorphic to a subspace of $\CC([0,1]) $, it follows from the previous theorem and the separability of $JT$ that $\CC([0,1])$ does not have the WLP. 
\end{proof}

\begin{CORO}
Let $K$ be a compact Hausdorff space.
\begin{enumerate}
\item If $K$ is metrizable, then $\CC(K)$ has the WLP if and only if $K$ is countable. 
\item If $\CC(K)$ has the WLP then $K$ is scattered. The converse is not true since $c_0 (\mathfrak{c})$ does not have the WLP (Theorem \ref{c0lpWLP}) and it is isomorphic to a $\CC(K)$ space with $K$ scattered.
\end{enumerate}
\end{CORO}
\begin{proof}
If $K$ is a countable compact metric space, then $\CC(K)^\ast$ is separable \cite[Theorem 14.24]{fabian2011banach}, so $\CC(K)$ has the WLP (Theorem \ref{dualWLP}).
If $K$ is an uncountable compact metric space, then $\CC(K)$ is isomorphic to $\CC([0,1])$ \cite[Theorem 4.4.8]{opac-b1119539}, so $\CC(K)$ does not have the WLP (Corollary \ref{C[0,1]noWLP}).
Finally, if $K$ is not scattered, then $\CC(K)$ has a subspace isomorphic to $\CC([0,1])$ (see the proof of \cite[Theorem 14.26]{fabian2011banach}), so $\CC(K)$ does not have the WLP.
\end{proof}

\begin{REMA}
\label{remark}
Let $\lbrace X_i \rbrace_{i \in \Gamma} $ be a family of Banach spaces. Define $X:=(\bigoplus_{i \in \Gamma} X_i )_{\ell_p} $ with $1<p<\infty$ or $X:=(\bigoplus_{i \in \Gamma} X_i )_{c_0} $. If $f:[0,1] \rto X$ is a bounded function, then its set of points of weak discontinuity is
$E= \bigcup_{i \in \Gamma} E_i$, where each $E_i$ is the set of points of weak discontinuity of $f_i$ and $f_i$ is the i'th cordinate of $f$.
Thus, the countable $\ell_p$-sum or $c_0$-sum of Banach spaces with the WLP has the WLP.
We cannot extend this result to uncountable $\ell_p$-sums or $c_0$-sums even when $X_i=\R$ for every $i \in \Gamma$ (Theorem \ref{c0lpWLP}).

\end{REMA}

Now, we study the WLP for the spaces of the form $c_0(\kappa)$ and $\ell_p(\kappa)$ with $\kappa$ a cardinal.

\begin{THEO}
For any cardinal $\kappa$ and any $1<p<\infty$, $c_0(\kappa)$ has the WLP if and only if $\ell_p(\kappa)$ has the WLP.
\label{Teosii}
\end{THEO}

\begin{proof}

If $\ell_p (\kappa)$ does not have the WLP, then there exists a Riemann integrable function $f:[0,1] \rto \ell_p (\kappa)$ which is not weakly continuous a.e. If $I: \ell_p (\kappa) \rto c_0(\kappa)$ is the canonical inclusion, then the function $I \circ f$ is weakly continuous at a point $t \in [0,1] $ if and only if $f$ is weakly continuous at $t$ by Remark \ref{remark}. Therefore, $I \circ f$ is not weakly continuous a.e.
Since $I$ is an operator, $I \circ f$ is also Riemann integrable. Thus, $c_0( \kappa)$ does not have the WLP.

To prove the other implication, suppose $c_0 (\kappa)$ does not have the WLP. Then, there exists a Riemann integrable function $f:[0,1] \rto c_0 ( \kappa )$ which is not weakly continuous a.e. 
Let $f_\alpha$ be the $\alpha$'th cordinate of $f$ for every $\alpha \in \kappa$ and $E_\alpha^n$ be the set of points where $f_\alpha$ has oscillation strictly bigger than $\frac{1}{n}$ for every $n \in \N$. Note that each $E_\alpha^n$ has Lebesgue measure zero. Since $f$ is not weakly continuous a.e., $\bigcup_{\alpha < \kappa} \( \bigcup_{n \in \N} E_\alpha ^n \)$ has not Lebesgue measure zero, so there exists $n \in \N$ such that $\bigcup_{\alpha < \kappa} E_\alpha ^n$ has not Lebesgue measure zero.

Set $F_0 := E_0^n$ and $F_\alpha := E_\alpha ^n \setminus \( \bigcup_{\beta < \alpha } E_\beta^n \)$ for every $\alpha \in \kappa \setminus \lbrace 0 \rbrace$. The sets $F_\alpha$ are pairwise disjoint. Let $\chi_{F_\alpha} \colon [0,1] \rto \lbrace 0,1 \rbrace$ be the characteristic function of $F_\alpha$ for every $\alpha < \kappa$ and $g \colon [0,1] \rto c_0 (\kappa) $ the function defined by the formula
$g(t)= \sum_{\alpha < \kappa} \chi_{F_\alpha} (t) e_\alpha $ for every $t \in [0,1]$,
where $\lbrace e_\alpha \rbrace _{\alpha < \kappa}$ is the canonical basis of $X$. 

Notice that $g$ is not weakly continuous a.e.~since each $\chi_{F_\alpha}$ is not continuous at any point of $F_\alpha$ (because $\mu(F_\alpha)=0$) and $\bigcup_{\alpha < \kappa} F_\alpha= \bigcup_{\alpha < \kappa} E_\alpha ^n$ is not Lebesgue null. 
We claim that $g$ is Riemann integrable. Let $\varepsilon >0$. Since $f$ is Riemann integrable, there exists a partition $\mathcal{P}_\varepsilon$ of $[0,1]$ such that $\|f( \mathcal{P}_1) - f(\mathcal{P}_2) \| < \frac{\varepsilon}{n} $ for all tagged partitions $\mathcal{P}_1$ and $\mathcal{P}_2$ of $[0,1]$ that have the same intervals as $\mathcal{P}_\varepsilon$. 
For every $\alpha < \kappa$ and any tagged partitions $\mathcal{P}_1$ and $\mathcal{P}_2$ of $[0,1]$ that have the same intervals as $\mathcal{P}_\varepsilon$, 
$$ | \chi_{F_\alpha} ( \mathcal{P}_1 ) - \chi_{F_\alpha} ( \mathcal{P}_2 ) | \leq  \sum_{i=1}^N \mu (I_i)  \leq 
n | f_\alpha( \mathcal{P}_1') - f_\alpha ( \mathcal{P}_2' ) | \leq n \|f( \mathcal{P}_1') - f(\mathcal{P}_2') \| < \varepsilon $$
for suitable tagged partitions  $\mathcal{P}_1'$ and $\mathcal{P}_2'$ of $[0,1]$ with the same intervals as $\mathcal{P}_\varepsilon$, where $I_1, I_2,\dots, I_N$ are the intervals of $\mathcal{P}_\varepsilon $ whose interior has non-empty intersection with $E_\alpha^n$.

Therefore, $g$ is Riemann integrable. Let $h:[0,1] \rto \ell_p (\kappa)$ be the function given by the formula
$h(t)=\sum_{\alpha < \kappa} \chi_{F_\alpha} (t) e_\alpha $. Since the sets $F_\alpha$ are pairwise disjoint, the function $h$ is well-defined. Moreover, $h$ is not weakly continuous a.e. because $I \circ h= g$. Set $F=\bigcup_{\alpha < \kappa} F_\alpha $ and $\phi : F \rto \kappa $ such that $\phi (t) = \alpha $ if $t \in F_\alpha$.
We claim that $h$ is Riemann integrable with integral zero. Let $\varepsilon >0$ and $\mathcal{P}_\varepsilon = \lbrace I_1, I_2, \dots, I_M \rbrace$ be a partition of $[0,1]$ such that $\| g( \mathcal{P}') \| < \varepsilon $ for any tagged partition $\mathcal{P}'$ of $[0,1]$ with the same intervals as $\mathcal{P}_\varepsilon$. Notice that 
\begin{equation} \label{eq1}
\mu \Biggl( \bigcup_{\substack{\Int (I_i)\cap F_\alpha \neq \emptyset}} I_i \Biggr) < \varepsilon \mbox{ for every } \alpha < \kappa .
\end{equation}
Thus, for any tagged partition $\mathcal{P}=\lbrace (s_i, I_i)\rbrace_{i=1}^M$ the following inequalities hold: 
 
$$ \| h( \mathcal{P} ) \| = \norm[\bigg]{ \sum_{ s_i \in F } \mu(I_i) e_{\phi(s_i)}  } =\norm[\bigg]{ \sum_{\alpha < \kappa} \mu \biggl( \bigcup_{\phi(s_i)=\alpha} I_i \biggr)  e_\alpha  }= $$ $$= \biggl( \sum_{\alpha < \kappa} \mu \biggl( \bigcup_{\phi(s_i)= \alpha} I_i \biggr) ^p \biggr) ^\frac{1}{p} = \biggl( \sum_{\alpha < \kappa}  \mu \biggl( \bigcup_{\phi(s_i)=\alpha} I_i \biggr)^{p-1}  \mu \biggl( \bigcup_{\phi(s_i)=\alpha} I_i \biggr)   \biggr)^\frac{1}{p} \leq $$ $$ \overset{(\ref{eq1})}{\leq} \varepsilon^\frac{p-1}{p} \biggl( \sum_{\alpha < \kappa}  \mu \biggl( \bigcup_{\phi(s_i)=\alpha} I_i \biggr)  \biggr)^\frac{1}{p} \leq \varepsilon^\frac{p-1}{p} $$
Therefore, $h$ is Riemann integrable with Riemann integral zero.
\end{proof}

The LP is separably determined \cite{Pizzotti89}. Nevertheless, it follows from the following theorem that the WLP is not separably determined, since every separable infinite-dimensional subspace of $\ell_2 (\kappa) $ is isomorphic to $\ell_2$ (which has separable dual).

\begin{THEO}
\label{c0lpWLP}
Let $\kappa$ be a cardinal and $X=c_0(\kappa)$ or $X=\ell_p(\kappa)$ with $1<p<\infty$. 
\begin{enumerate}
\item If $ \kappa < \covM $ then $X$ has the WLP.
\item If $ \kappa \geq \nonSN $ then $X$ does not have the WLP.
\end{enumerate}
\end{THEO}

\begin{proof}

It is enough to prove the result when $X=c_0(\kappa)$ due to Theorem \ref{Teosii}.
Since $c_0(\kappa)^\ast = \ell_1 (\kappa)$ has density character $\kappa$, it follows from Theorem \ref{dualWLP} that $c_0(\kappa)$ has the WLP  if $\kappa < \covM$.

Suppose $\nonSN \leq \kappa \leq \mathfrak{c}$. Due to Theorem \ref{Pawlikowski}, there exist a closed Lebesgue null set $F $ and a set $E=\lbrace x_\alpha \rbrace_{\alpha < \kappa} $ in $\R$ such that $E+F$ does not have Lebesgue measure zero.
Without loss of generality, we may assume that $E,F \subset [0,1]$ and $\( E+F \) \cap [0,1]$ does not have Lebesgue measure zero.
Set $F_0 := \( {x_0}+F \) \cap [0,1]$ and $F_\alpha := \( \({x_\alpha}+F \) \cap [0,1] \) \setminus \( \bigcup_{\beta < \alpha } F_\beta \)$ for every $0 < \alpha < \kappa$. Let $\chi_{F_\alpha} \colon [0,1] \rto \lbrace 0,1 \rbrace$ be the characteristic function of $F_\alpha$ for every $\alpha < \kappa$ and $f \colon [0,1] \rto X $ the function defined by the formula
$f(t)= \sum_{\alpha < \kappa} \chi_{F_\alpha} (t) e_\alpha $ for every $t \in [0,1]$,
where $\lbrace e_\alpha \rbrace _{\alpha < \kappa}$ is the canonical basis of $c_0(\kappa)$.

Since the sets $F_\alpha $ are pairwise disjoint, the function $f$ is well-defined.
Each $\chi_{F_\alpha}$ is not continuous at $F_\alpha$, since $F_\alpha$ cannot contain an interval of $[0,1]$.
Therefore, $f$ is not weakly continuous a.e. because $\bigcup_{\alpha < \kappa} F_\alpha = \(E+F\) \cap [0,1]$ does not have Lebesgue measure zero. 

We claim that $f$ is Riemann integrable. For every $\alpha < \kappa$ and every tagged partition $\mathcal{P} = \lbrace (s_i, I_i) \rbrace_{i=1}^N$ we have 
$$  \chi_{F_\alpha} ( \mathcal{P} )  = \sum_{i=1}^N \mu(I_i) \chi_{F_\alpha}(s_i) \leq \sum_{i=1}^N \mu(I_i-x_\alpha) \chi_{F}(s_i-x_\alpha) = \chi_F (\mathcal{P'} ) $$
for a suitable tagged partition $\mathcal{P'}$ with $\| \mathcal{P} \| = \| \mathcal{P'} \|$.
Since $F \subset [0,1]$ is a closed Lebesgue measure zero set, the characteristic function $\chi_F$ is Riemann integrable due to Lebesgue's Theorem. Then, for every $\varepsilon >0$ there exists $\delta>0$ such that $\chi_F( \mathcal{P} ) < \varepsilon $ for every tagged partition $\mathcal{P}$ with $\| \mathcal{P} \| < \delta$. 
Therefore, for every $\varepsilon >0$ there exists $\delta>0$ such that $\chi_{F_\alpha}( \mathcal{P}) < \varepsilon $ for all tagged partitions $\mathcal{P}$ with $\| \mathcal{P} \| < \delta$ and for every $\alpha < \kappa$.
Thus, $f$ is Riemann integrable since $\|f( \mathcal{P}) \|= \sup_{\alpha < \kappa} \chi_{F_\alpha}( \mathcal{P})  < \varepsilon $ for every tagged partition $\mathcal{P}$ of $[0,1]$ with $\| \mathcal{P} \| < \delta$.   
\end{proof}

The facts that the countable $\ell_1$-sum of spaces with the WLP has the WLP (Theorem \ref{WLPcountable}) and that $L^1(\lambda)$ has the WLP if $\dens (L^1(\lambda))< \covM$ (Theorem \ref{WLPL1}) will be a consequence of the following lemma.

\begin{LEMM}

\label{lemaprincipalaux}

Let $(\Omega,\Sigma,\lambda)$ be a probability space and $\mathfrak{P}=\{P_A: A\in \Sigma\}$
a family of operators on a Banach space $X$ such that
\begin{enumerate}
\item[(1)] $P_A+P_{\Omega \setminus A}=P_\Omega=id_X$ for every $A\in \Sigma$;
\item[(2)] $\|P_A(x)\| \leq \|x\|$ for every $x\in X$ and every $A \in \Sigma$;
\item[(3)] $\|P_A(x)\|+\|P_B(x')\| \leq \max\{\|x+x'\|,\|x-x'\|\}$ for every $x,x'\in X$ whenever $A\cap B=\emptyset$;
\item[(4)] $\lim_{\lambda(A)\to 0}\|P_A(x)\|=0$ for every $x\in X$.
\end{enumerate}

Let $f:[0,1]\to X$ be a Riemann integrable function. Then there is a measurable set $E \sub [0,1]$ with $\mu(E)=1$
such that, for every sequence $\Suc{t_n}{n}$ in~$[0,1]$ converging to some $t\in E$, the set $\{f(t_n):n\in \N\}$ is 
$\mathfrak{P}$-uniformly integrable, in the sense that
$$
	\lim_{\lambda(A)\to 0}\sup_{n\in \N} \big\|P_A(f(t_n))\big\|=0.
$$
\end{LEMM}

\begin{proof}

The proof is similar to that of \cite[Lemma 2.3]{MR2591586} and \cite[Lemma 3]{Wang2001}.
Fix $\beta >0$ and denote by $E_\beta$ the set of points $t \in [0,1]$ such that for every $\delta >0 $ there exist $t' \in [0,1]$  with $|t'-t| < \delta $ and a set $A \in \Sigma$ with $\lambda(A)< \delta$ such that $$\| P_A (f(t) - f(t'))\| > \beta.$$ 
Let $\mu^\ast$ be the Lebesgue outer measure in $[0,1]$. We show that $\mu^\ast ( E_\beta)=0$ with a proof by contradiction. Suppose $\mu^\ast ( E_\beta)>0$. Since $f$ is Riemann integrable, we can choose a partition $\PP=\lbrace J_1,\dots, J_m \rbrace$ of $[0,1]$ such that
\begin{equation} \label{**}
\norm[\bigg]{ \sum \limits_{j=1}^m \mu(J_j)(f(\xi_j)-f(\xi_j')) } < \beta \mu^\ast (E_\beta) 
\end{equation}
for all choices $\xi_j, \xi_j' \in J_j, 1 \leq j \leq m$. Let $S= \lbrace j \in \lbrace 1,\dots,m \rbrace: I_j \cap E_\beta \neq \emptyset \rbrace$, where $I_j = \Int (J_j)$ for each $j=1,\dots,m$. Thus, 
\begin{equation} \label{***}
\sum_{j \in S} \mu^\ast (I_j \cap E_\beta) = \mu^\ast ( E_\beta).
\end{equation}
It is not restrictive to suppose $S=\lbrace 1,\dots,n \rbrace $ for some $1 \leq n \leq m$.

Because of the definition of $ E_\beta $ and $I_1$, there exist points $t_1 \in I_1 \cap E_\beta $ and $ t_1' \in I_1$ such that $ \| f(t_1) - f(t_1') \|\geq \| P_A (f(t_1)-f(t_1'))\|  > \beta $ for some $A\in \Sigma$, hence $\| \mu(I_1) (f(t_1)- f(t_1')) \| > \beta \mu(I_1) $. 

Fix $1 \leq k < n$ and assume that we have already chosen points $t_j ,t_j' \in I_j$ for all $1\leq j \leq k$ with the property that $$ \norm[\bigg]{ \sum \limits_{j=1}^k \mu (I_j) (f(t_j) - f(t_j')) } > \beta \biggl( \sum \limits_{j=1}^k \mu(I_j) \biggr).$$
	
Define $x:= \sum \limits_{j=1}^k \mu(I_j) ( f(t_j)-f(t_j')) \in X$ and 
$$ \alpha := \| x \| - \beta \biggl( \sum \limits_{j=1}^k \mu(I_j)\biggr) > 0.$$
Due to (4), we can choose $\delta >0$ such that $  \| P_A (x) \| <\alpha$ whenever $A\in \Sigma$ satisfies $\lambda(A) <\delta$. Take $t_{k+1}, t_{k+1}' \in I_{k+1}$ and a set $A\in \Sigma $ with $\lambda(A)<\delta$ such that 
$ \| P_A (f(t_{k+1}) - f(t_{k+1}'))\| > \beta $,
so $y:= \mu( I_{k+1})(f(t_{k+1})- f(t_{k+1}'))$ satisfies $$ \| P_A(y)\| > \beta \mu (I_{k+1}).$$

By the choice of $A$, (1) and (3), we also have (interchanging the role of $t_{k+1}$ and $t_{k+1}'$ if necessary)
$$ \norm[\bigg]{ \sum \limits_{j=1}^{k+1} \mu (I_j) (f(t_j)-f(t_j')) } \geq \| P_A(y) \|+ \| P_{A^c}(x) \| \geq \| P_A(y) \|+ \| x\| - \| P_{A}(x) \| > $$
$$ > \beta\mu(I_{k+1}) + \alpha + \beta \sum \limits_{j=1}^{k} \mu(I_j) - \|P_{A}(x)\| >  \beta \sum \limits_{j=1}^{k+1} \mu(I_j) .$$
 
Thus, there exist $t_j, t_j' \in I_j$ for all $1 \leq j \leq n $ such that 
$$ \norm[\bigg]{ \sum \limits_{j=1}^n \mu(I_j)(f(t_j) - f(t_j')) } > \beta \biggl( \sum \limits_{j=1}^n \mu(I_j) \biggr) \overset{(\ref{***})}{\geq} \beta \mu^\ast (E_\beta),$$
which contradicts the inequality (\ref{**}).
So we can conclude that $\mu^\ast( E_\beta) = 0$.

Therefore, $E:= [0,1] \setminus \bigcup_{n \in \N} E_{\frac{1}{n}}$ is measurable with $\mu(E)=1$.
Fix $t\in E$ and $m \in \N$. Since $t \notin E_{\frac{1}{m}}$, there exists $\delta_m >0$ such that for every $t' \in [0,1]$  with $|t'-t| < \delta_m $ and every set $A \in \Sigma$ with $\lambda(A)< \delta_m$, $$\| P_A (f(t) - f(t'))\| \leq \frac{1}{m}.$$ 
Thus, for every $m\in \N$, every sequence $\Suc{t_n}{n}$ converging to $t$ and every $A\in \Sigma$ with $\lambda (A) < \delta_m$, $$\|P_A( f(t_n))\| \leq \|P_A (f(t)) \| + \frac{1}{m} \mbox{ for } n \mbox{ big enough depending only on }m. $$ Now the conclusion follows from (4).
\end{proof}

Let $\lbrace X_i \rbrace_{i \in \Gamma }$ be a family of Banach spaces. We denote by $\pi_j: (\bigoplus_{i \in \Gamma} X_i) \rto X_j $ the canonical projection onto $X_j$ for each $j \in \Gamma$.

We will need the following property of $\ell_1$-sums and the space $L_1(\lambda)$ for Theorems \ref{WLPcountable} and \ref{WLPL1}:

\begin{LEMM} 
\label{lemaauxiliar}
Let $(\Omega, \Sigma, \lambda)$ be a probability space and $\lbrace X_i \rbrace_{i \in \Gamma} $ a family of Banach spaces. Then:
\begin{enumerate}
\item $\max \lbrace \|x+y\|,\|x-y\| \rbrace \geq \sum_{i \in A} \| \pi_i(x) \| + \sum_{i \in B} \|\pi_i(y) \|$ for every vectors $x,y \in (\bigoplus_{i \in \Gamma } X_i )_{\ell_1}$  and any disjoint sets $A,B \subset \Gamma$.
\item $\max \lbrace \|f+g\|,\|f-g\| \rbrace \geq \int_A |f|d\lambda + \int_B |g|d\lambda$ for any $f,g \in L_1(\lambda)$ and any disjoint sets $A,B \in \Sigma$.
\end{enumerate}
\end{LEMM}

\begin{proof}
The second part is essentially Lemma 2 of \cite{Wang2001}.
The proof of the first part is analogous and we include it for the sake of completeness.
Let $x,y \in (\bigoplus_{i \in \Gamma } X_i )_{\ell_1}$, $A,B \subset \Gamma$ be disjoint sets and $\varepsilon >0$. Without loss of generality, we may assume that $A= \lbrace a_n \colon n \in \N \rbrace$ and $B= \lbrace b_n \colon n \in \N \rbrace$ are countable subsets. Consider the functionals $x^\ast, y^\ast \in (\bigoplus_{i \in \Gamma } X_i )_{\ell_1}^\ast = (\bigoplus_{i \in \Gamma } X_i^\ast )_{\ell_\infty}$ defined by 
$ x^\ast(u) = \sum_{i \in A} x_i^\ast(\pi_i(u)) $ and 
$ y^\ast(u) = \sum_{i \in B} y_i^\ast(\pi_i(u)) $ for every $u \in (\bigoplus_{i \in \Gamma } X_i )_{\ell_1}$, where each $x_i ^\ast, y_i^\ast \in X_i^\ast$ satisfies $\|x_i^\ast\|=\|y_i^\ast\|=1$, $x_i ^\ast( \pi_i(x)) = \| \pi_i(x) \|$ if $i=a_n$ and $y_i ^\ast( \pi_i(y)) = \| \pi_i(y) \|$ if $i=b_n$.
Then, since $A,B$ are disjoint, $\|x^\ast+y^\ast\| =\|x^\ast-y^\ast\| = 1$. Therefore,
$$ \|x+y \| + \| x-y\| \geq \langle x+y, x^\ast+y^\ast \rangle + \langle x-y, x^\ast - y^\ast \rangle= 2 \langle x, x^\ast \rangle + 2 \langle y, y^\ast \rangle = $$
$$ = 2 \biggl( \sum_{i \in A} x_i^ \ast( \pi_i(x)) + \sum_{i \in B} y_i^\ast( \pi_i(y)) \biggr) = 2 \biggl(\sum_{i \in A} \| \pi_i(x) \| + \sum_{i \in B} \|\pi_i(y) \| \biggr),$$
so $\max \lbrace \|x+y\|,\|x-y\| \rbrace \geq \sum_{i \in A} \| \pi_i(x) \| + \sum_{i \in B} \|\pi_i(y) \|$.
\end{proof}

\begin{THEO}

\label{WLPcountable}

Let $\lbrace X_i \rbrace_{i \in \N} $ be Banach spaces with the WLP. Then the space $X:=(\bigoplus_{i \in \N} X_i )_{\ell_1} $ has the WLP.

\end{THEO}
\begin{proof}
We are going to apply Lemma \ref{lemaprincipalaux}.
Take $\Omega:=\N$, $\Sigma:=\mathcal{P}(\N)$ the power set of $\N$,
$\lambda(A):=\sum_{n\in A}2^{-n}$ and $\mathfrak{P}=\{P_A: A\in \Sigma\}$ with 
$$ \pi_i (P_A(x))= \begin{cases}
\pi_i(x) & \mbox{ if } i \in A \\
0 & \mbox{ if } i \notin A
\end{cases}
$$
for every $A \in \Sigma$ and every $x \in X$. Property (3) of Lemma \ref{lemaprincipalaux} is Lemma \ref{lemaauxiliar}(1) and property (4) holds because if $\lambda(A) < \frac{1}{2^n}$, then $A \subset \lbrace n, n+1, \dots \rbrace$, so $$ \| P_A (x) \| = \sum_{i \in A} \| \pi_i (x) \| \leq \sum_{i \geq n } \| \pi_i (x) \| $$
for every $x \in X$.
Therefore, we can apply Lemma \ref{lemaprincipalaux}, so there exists a measurable set $E \subset [0,1]$ with $\mu(E)=1$ such that for every sequence $\Suc{t_n}{n}$ in $[0,1]$ converging to some $t \in E$ the set $\lbrace f(t_n): n \in \N \rbrace$ is $\mathfrak{P}$-uniformly integrable. We can assume that, for each $i\in \N$, the 
map $t \mapsto \pi_i(f(t))$ is weakly continuous at each point of~$E$ because each $X_i$ has the WLP.

It is a well known fact that a sequence $\Suc{x_n}{n}$ in~$X$ converges weakly to~$x\in X$ if and only if it satisfies the following two conditions:
\begin{enumerate}
\item[(i)] $\pi_i(x_n) \to \pi_i(x)$ weakly in~$X_i$ for every $i\in \N$;
\item[(ii)] for every $\varepsilon>0$ there is a finite set $J \sub \N$ such that $\sup_{n\in \N}\|P_{\N \setminus J}(x_n)\|\leq \varepsilon$.
\end{enumerate}

Since $\mathfrak{P}$-uniform integrability is equivalent to (ii), 
it follows that $f$ is weakly continuous at each point of $E$.
\end{proof}

\nocite{MiragliaRocha84}

A similar idea to that of Theorem \ref{WLPcountable} let us prove the following theorem, which improves 
\cite[Theorem 5]{Wang2001} and \cite[Proposition 2.10]{MR2591586}.

\begin{THEO}

\label{WLPL1}

Let $(\Omega, \Sigma, \lambda)$ be a probability space.
\begin{enumerate}
\item  If $\dens(L^1(\lambda))<\covM$ then $L^1 ( \lambda )$ has the WLP.
\item  If $\dens(L^1(\lambda))\geq \nonSN$ then $L^1 ( \lambda )$ does not have the WLP.
 \end{enumerate}
\end{THEO}

\begin{proof}
Fix a Riemann integrable function $f \colon [0,1] \rto L^1(\lambda)$.
Take $P_A(x):=x\chi_A$ for every $A \in \Sigma$ and every $x \in L^1(\lambda)$. The family of operators $\lbrace P_A \colon A \in \Sigma \rbrace $ fulfills the requirements of Lemma \ref{lemaprincipalaux} (bear in mind Lemma \ref{lemaauxiliar}). Then $\mathfrak{P}$-uniform integrability is the
usual uniform integrability and therefore a set is bounded and $\mathfrak{P}$-uniformly integrable if and only if it is relatively weakly compact due to Dunford's Theorem (see \cite[Theorem 5.2.9]{opac-b1119539}). 
Lemma \ref{lemaprincipalaux} ensures that there exist a measurable set $E \subset [0,1]$ with $\mu(E)=1$ such that for every sequence $\Suc{t_n}{n}$ in $[0,1]$ converging to some $t \in E$, the set $\lbrace f(t_n) \colon n \in \N \rbrace$ is relatively weakly compact.

Let $\CC \subset \Sigma $ be a dense family of $\lambda$-measurable sets, i.e. such that $$\inf_{C \in \CC} \lambda(A \bigtriangleup C) =0 \mbox{ for every } A \in \Sigma .$$

Let $\Suc{h_n}{n}$ be a relatively weakly compact
sequence in $L^1(\lambda)$ and $h\in L^1(\lambda)$. Since $\CC$ is a dense family of $\lambda$-measurable sets, if $\int_C h_n \, d\mu \to \int_C h \, d\mu$ for every $C\in \mathcal{C}$, then 
$h=\w$-$\lim h_n$.

Suppose $\dens(L^1 ( \lambda ))<\covM$. Then $\CC$ can be taken such that $|\CC| < \covM$. Therefore, we can assume that, for each $C\in \mathcal{C}$, the Riemann integrable
map $t \mapsto \int_C f(t) \, d\lambda$ is continuous at each point of $E$.
Then, for every sequence $\Suc{t_n}{n}$ in~$[0,1]$ converging to
a point~$t\in E$, we have $f(t)=\w$-$\lim f(t_n)$. 

Now suppose $\nu= \dens(L^1(\lambda)) \geq \nonSN$. Due to Maharam's Theorem (see \cite[p.~127, Theorem 9]{lacey2012isometric}),  $L^1(\lambda) $ contains an isometric copy of $L^1(\mu_\nu)$, where $\mu_\nu$ is the usual product probability measure on $\lbrace 0,1 \rbrace ^\nu$. Since $L^1(\mu_\nu)$ contains an isomorphic copy of $\ell_2( \nu)$ (see \cite[p.~128, Theorem 12]{lacey2012isometric}) and $\ell_2(\nu)$ does not have the WLP (Theorem \ref{c0lpWLP}), we conclude that $L^1(\lambda)$ does not have the WLP.
\end{proof}

Theorem \ref{WLPcountable} can be extended to arbitrary $\ell_1$-sums:

\begin{THEO}

\label{l1sumaWLP}

The arbitrary $\ell_1$-sum of a family of Banach spaces with the WLP has the WLP. 

\end{THEO}

\begin{proof}

The proof uses some ideas of \cite{MiragliaRocha84}.
Let $f:[0,1] \rto X:=(\bigoplus_{i \in \Gamma } X_i )_{\ell_1}$ be a Riemann integrable function, where $\lbrace X_i \rbrace_{i \in \Gamma} $ is a family of Banach spaces with the WLP.
For each $J \subset \Gamma$, we denote by $P_J : X \rto X$ the function defined by $\pi_i \left( P_J (x) \right) =\pi_i(x)$ if $i \in J$ and $\pi_i \left( P_J(x) \right)=0$ in any other case. 
Let $\Suc{r_n}{n}$ be an enumeration of the rational numbers in $[0,1]$ and fix a countable set $L \subset \Gamma$ such that $ P_L(f(r_n))=f(r_n)$ for every $n \in \N$.
Then, $f=(f-P_L f)+P_L f$. Since $P_L f$ is Riemann integrable and takes values in the space 
$$X|_L :=\lbrace x \in X: \pi_i(x)=0 \mbox{ for each } i \notin L \rbrace,$$ which is isomorphic to a countable $\ell_1$-sum of spaces with the WLP, by Theorem \ref{WLPcountable} $P_L f$ is weakly continuous almost everywhere.

Therefore, we can assume that $\int_0^1 f(t)dt =0$ and that $f$ is null over a dense set.
Let 
$$ A_n^J := \lbrace t \in [0,1] \colon \| P_{J^c}(f(t)) \| \geq \frac{1}{n} \rbrace $$
 for each $n \in \N$ and each subset $J \subset \Gamma$.
If $J_1 \subset J_2 \subset \Gamma$, then $A_n ^{J_2} \subset A_n^{J_1}$.

\textbf{Claim:} \textit{For every $n \in \N$ there exists a countable set $J \subset \Gamma$ with $\mu\( \overline{A_n^J}\)=0$.}
Suppose this is not the case. Then, there exist $n\in \N$ and $\delta >0$ with $\mu \( \overline{A_n^J}\)>\delta$ for every countable subset $J \subset \Gamma$ (if for every $m \in \N$ we can take a countable set $J_m \subset \Gamma$ with $\mu \( \overline{A_n^{J_m}}\)<\frac{1}{m}$, then $J=\bigcup_{m \in \N} J_m$ verifies $\mu \( \overline{ A_n^J } \)=0$).
Let $\PP= \lbrace I_1,I_2,\dots,I_N \rbrace $ be a partition of $[0,1]$ such that
\begin{equation}\label{estrella}
 \norm[\bigg]{ \sum \limits_{j=1}^N \mu(I_j)(f(\xi_j)-f(\xi_j')) } < \frac{\delta}{n} \mbox{ for all choices } \xi_j, \xi_j' \in I_j, 1 \leq j \leq N.
\end{equation}
Let $J \subset \Gamma$ be a countable subset. Since $\sum \limits_{j=1}^N \mu  \(I_j \cap \overline{A_n ^J}\)= \mu \(  \overline{A_n ^J} \) > \delta$ and $f$ is null over a dense set, we can suppose that there exist $\xi_1 \in \Int (I_1) \cap A_n ^J$ and $ \xi_1' \in I_1$ such that $ \| \mu(I_1)(f(\xi_1)-f(\xi_1')) \|  \geq \frac{1}{n} \mu(I_1) $. Let $J_1= \supp{ f(\xi_1)} \cup \supp{ f(\xi_1')}$.
By (\ref{estrella}) we have $\mu(I_1) < \delta < \sum_{j=1}^N \mu \(I_j \cap \overline{A_n^{J_1}}\)$ and so it is not restrictive to suppose $ \Int (I_2) \cap \overline{A_n ^{J_1}} \neq \emptyset$. Thus, due to Lemma \ref{lemaauxiliar}, we can choose $ \xi_2, \xi_2' \in I_2 $ such that 
$$\| \mu(I_1)(f(\xi_1)-f(\xi_1'))  + \mu(I_2)(f(\xi_2)-f(\xi_2'))  \| \geq \frac{1}{n}( \mu(I_1)+ \mu(I_2)).$$
Fix $1 \leq k < N$ and assume that we have already chosen points $\xi_j, \xi_j' \in I_j $ for all $1\leq j \leq k$ with the property that $$ \norm[\bigg]{ \sum \limits_{j=1}^k \mu (I_j)( f(\xi_j)-f(\xi_j')) } \geq \frac{1}{n} \( \sum \limits_{j=1}^k \mu(I_j)\).$$
Set $J_k:= \bigcup_{j=1}^k \supp{f(\xi_j)}\cup\supp{f(\xi_j')}$, which is countable. By (\ref{estrella}) we have  $\sum \limits_{j=1}^k \mu(I_j) < \delta < \sum_{j=1}^N \mu \( I_j \cap \overline{ A_n^{J_k}} \)$, hence it is not restrictive to suppose that $\Int(I_{k+1}) \cap \overline{A_n ^{J_k}} \neq \emptyset$ and therefore that there exist points $\xi_{k+1}, \xi_{k+1}' \in  I_{k+1}$ such that $$ \norm[\bigg]{ \sum \limits_{j=1}^{k+1} \mu (I_j) (f(\xi_j)-f(\xi_j')) } \geq \frac{1}{n} \( \sum \limits_{j=1}^{k+1} \mu(I_j)\).$$
Since $ \sum \limits_{j=1}^{N} \mu(I_j)=1 > \delta$, it follows that there exist $\xi_j, \xi_j' \in I_j$ for every $1 \leq j \leq N$ such that $$ \norm[\bigg]{ \sum \limits_{j=1}^{N} \mu (I_j) (f(\xi_j)-f(\xi_j')) } \geq \frac{\delta}{n}.$$ But this is a contradiction with (\ref{estrella}). Therefore, the \textbf{Claim} is proved.

Thus, for every $n \in \N$ there exists a countable set $J_n$ such that $\mu \( \overline{ A_n ^{J_n}} \)=0$. Fix $J:= \bigcup_{n \in \N} J_n$.
Theorem \ref{WLPcountable} guarantees the existence of a set $F \subset [0,1]$ of measure one such that $P_J(f)$ is weakly continuous at every point of $F$.
Let $E=F\setminus \bigcup_{n \in \N} \overline{A_n^J}$. Then, $\mu(E)=1$, $f=P_J(f)+P_{J^c}(f)$, $P_J(f)$ is weakly continuous at each point of $E$ and $P_{J^c}(f)$ is norm continuous at each point of $E$ (if $t_n \rto t \in E$, then, for every $m \in \N$, $ t_n \notin A_m^J $ for $n$ big enough so $\|P_{J^c}(f) (t_n) \| < \frac{1}{m}$).
\end{proof}

\begin{CORO}[\cite{Pizzotti89, zbMATH03406284}]
$\ell_1 (\kappa)$ has the LP for any cardinal $\kappa$.
\end{CORO}

\begin{proof}
Since $\ell_1( \kappa)$ has the Schur property, $\ell_1(\kappa)$ has the LP if and only it has the WLP. 
Therefore, the conclusion follows from Theorem \ref{l1sumaWLP}.
\end{proof}

As an application of \ref{l1sumaWLP} we also obtain the following result:

\begin{CORO}
\label{dualC(K)WLP}

Let $K$ be a compact Hausdorff space. Then, $\CC(K)^\ast$ has the WLP if $\dens ( L^1( \lambda)) < \covM$ for every regular Borel probability $\lambda$ on $K$.

\end{CORO}

\begin{proof}

For every compact Hausdorff space $K$, the Banach space $\CC(K) ^\ast$ is isometric to a $\ell_1$-sum of $L^1( \lambda)$ spaces, where each $\lambda$ is a regular Borel probability measure on $K$ (see the proof of \cite[Proposition~4.3.8]{opac-b1119539}).
Thus, $\CC(K)^\ast$ has the WLP if each space $L^1( \lambda)$ has the WLP, due to Theorem \ref{l1sumaWLP}.
Hence, the result follows from Theorem \ref{WLPL1}. 
\end{proof}

\begin{CORO}
\label{COROMS}
If $K$ is a compact Hausdorff space in the class $MS$ (i.e. $L^1(\lambda)$ is separable for every regular Borel probability on $K$), then $\CC(K)^\ast$ has the WLP.
\end{CORO}

Some classes of compact spaces in the class $MS$ are metric compacta, Eberlein compacta, Radon-Nikod\'{y}m compacta, Rosenthal compacta and scattered compacta. For more details on this class, we refer the reader to \cite{DzamonjaKunen}, \cite{MarciszewskiPlebanek} and \cite{PlebanekSobota}.

The LP is a three-space property, i.e. if $X$ is a Banach space and $Y$ is a subspace of $X$ such that $Y$ and $X/Y$ have the LP, then $X$ has the LP  \cite[Proposition 1.19]{Pizzotti89}.  This result follows from Michael's Selection Theorem.
However, as far as we are concerned, it is not known whether the WLP is a three-space property. We have a positive result in the following case:

\begin{THEO}
Let $X$ be a Banach space and $Y$ a subspace of $X$. If $Y$ is reflexive, $\dens(Y) < \covM$ and $X/Y$ has the WLP, then $X$ has the WLP.
\end{THEO}

\begin{proof}

Let $Q:X \rto X/Y$ be the quotient operator and $\phi : X/Y \rto X$ be a norm-norm continuous map such that $Q \phi = 1_{X/Y}$ given by Michael's Selection Theorem (see \cite[Section 7.6]{fabian2011banach}). 
Let $f \colon [0,1] \rto X$ be a Riemann integrable function. Then, since $Qf$ is Riemann integrable and $X/Y$ has the WLP, there exists a set $E \subset [0,1]$ with $\mu(E)=1$ such that $Qf$ is weakly continuous at every point of $E$.
Set 
\begin{equation}\label{estrella2}
C = \lbrace x \in X \colon \exists \Suc{t_n}{n} \mbox{ converging to some }t\in E \mbox{ with } x=\w \mbox{-}\lim f(t_n) \rbrace.
\end{equation} 

First we are going to see that $\dens(C) < \covM$. Let $x\in C$ and $\Suc{t_n}{n}$ as in (\ref{estrella2}). Then $Qx= \w $-$\lim Qf(t_n)=Qf(t)$. Therefore, $x= \phi(Qx) + (x- \phi(Qx))$ with $\phi(Qx) \in \phi(Qf(E))$ and $x-\phi(Qx) \in Y$.
Notice that $\phi(Qf(E))$ is separable because of the $\w$-separability of $Qf(E)$ and Mazur's Lemma.  
Thus, $C \subset \phi(Qf(E))+Y$ satisfies $\dens(C)<\covM$. 

Let $\lbrace x_\alpha^\ast \rbrace_{\alpha \in \Gamma} \subset X^\ast$ be a set separating points of $C$ with $|\Gamma| < \covM$. Set $E_0 \subset E$ with $\mu(E_0)=1$ such that $x_\alpha^\ast \circ f$ is continuous at every point of $E_0$ for every $\alpha \in \Gamma$. Notice that this can be done because the set of discontinuity points of each $x_\alpha^\ast \circ f$ is an $F_\sigma$ Lebesgue null set and $|\Gamma| < \covM$.  
We claim that $f$ is weakly continuous at each point of $E_0$. Let $t \in E_0$ and $\Suc{t_n}{n}$ be a sequence converging to $t$. Since $Qf(t)= \w$-$\lim Qf(t_n)$, the set $\lbrace Qf(t_n) \colon n \in \N \rbrace$ is relatively weakly compact in $X/Y$. From the reflexivity of $Y$, it follows that $Q$ is a Tauberian operator, so $\lbrace f(t_n) \colon n \in \N \rbrace$ is relatively weakly compact in $X$ (see \cite[Theorem 2.1.5 and Corollary 2.2.5]{Gonzalez}).
Therefore, it is enough to prove the uniqueness of the limit of the subsequences of $\Suc{f(t_n)}{n}$. Let $x= \w$-$\lim \limits_k f(t_{n_k})$. Then, $x, f(t) \in C$ and $x_\alpha^\ast (x) = \lim \limits_k x_\alpha ^\ast (f(t_{n_k})) = x_\alpha ^\ast (f(t)) $ for every $\alpha \in \Gamma$, so $x=f(t)$.
\end{proof}

\section{Weak continuity does not imply integrability}

It is not true that every weakly continuous function is Riemann integrable \cite{MR0043366}. In fact, V. Kadets proved the following theorem:

\begin{THEO}[\cite{Kadets1994}]
\label{teoKadets}

If $X$ is a Banach space without the Schur property, then there is a weakly continuous function $f: [0,1] \rto X$ which is not Riemann integrable.

\end{THEO}

The proof of the previous theorem together with Josefson-Nissenzweig Theorem (see \cite[Chapter XII]{diestel1984sequences}) gives the following corollary:

\begin{CORO}
\label{coroKadets}

Given an infinite-dimensional Banach space $X$, there always exists a weak* continuous function $f:[0,1] \rto X^\ast$ which is not  Riemann integrable.

\end{CORO}

In \cite{Wang2000}, Wang and Yang extend the previous result to a general locally convex topology weaker than the norm topology.
In this section, we generalize these results in Theorem \ref{THEODP}.

Following the terminology used in \cite{EnfloGurariySeoane}, we say that a subset $M$ of a Banach space is spaceable if $M\cup \lbrace 0 \rbrace $ contains a closed infinite-dimensional subspace.

We start with the definitions of $\tau$-Dunford-Pettis operator and the $\tau$-Schur property, that coincide with the classical definitions of Dunford-Pettis or completely continuous operator and the Schur property when $\tau$ is the weak topology.

\begin{DEFI}

Let $X$ and $Y$ be Banach spaces and $\tau$ a locally convex topology on $X$ weaker than the norm topology. An operator $T:X \rto Y$ is said to be $\tau$-Dunford-Pettis ($\tau$-DP for short) if it carries bounded $\tau$-null sequences to norm null sequences.
A Banach space $X$ is said to have the $\tau$-Schur property if the identity operator $I: X \rto X$ is $\tau$-DP.

\end{DEFI}

\begin{THEO}

\label{THEODP}

Let $X$ and $Y$ be Banach spaces and $\tau$ be a locally convex topology on $X$ weaker than the norm topology.
If $T:X \rto Y$ is an operator which is not $\tau$-DP, 
then the family of all bounded $\tau$-continuous functions $f \colon [0,1]\rto X$ such that $Tf$ is not Riemann integrable is spaceable in $\ell_\infty([0,1], X)$, the space of all bounded functions from $[0,1]$ to $X$ with the supremum norm.

\end{THEO}

\begin{proof}

The proof uses ideas from \cite{Kadets1994}.
Since $T$ is not $\tau$-DP, we can take a bounded sequence $\Suc{x_n}{n}$ that is $\tau$-convergent to zero such that $\| Tx_n \| =1$ for all $n \in \N$.

Let $K \subset [0,1]$ be a copy of the Cantor set constructed by removing from $[0,1]$ an open interval $I_1^1$ in the middle of $[0,1]$ and removing open intervals $I_1^n , I_2^n,\dots I_{2^n}^n$ from the middles of the remaining intervals in each step.
Suppose that the removed intervals are so small that $\mu(K)> \frac{2}{3}$.
Let $\mathcal{C}_a([0,1])$ be the closed subspace of $\CC ([0,1])$ consisting of all continuous functions $g\colon [0,1] \rto \R$ antisymmetric with respect to the axe $x=\frac{1}{2}$ and with $g(0)=g(1)=0$.
For every $g \in \mathcal{C}_a([0,1])$ and every open interval $I=(a,b)$ in $[0,1]$, we define the functions $g_I \colon [0,1] \rto \R $ and $f_g \colon [0,1] \rto X$ as follows 
$$ g_I (t) = \begin{cases}
\ 0 & \mbox{ if } t  \notin (a,b),\\
\ g(\frac{t-a}{b-a}) & \mbox{ if } t\in [a,b].
\end{cases}
$$
$$ f_g (t) = \begin{cases}
\ 0 & \mbox{ if } t  \in K,\\
\ g_{I_k^n}(t) x_n  & \mbox{ if } t\in I_k^n.
\end{cases}
$$

The function $\phi \colon \CC_a([0,1]) \rto \ell_\infty ([0,1], X)$ given by the formula $\phi (g) :=f_g$ for every $g \in \CC_a ([0,1])$ is a linear map and satisfies $\| \phi (g) \| = (\sup_n \| x_n\| ) \|g \|$ for every $g \in \CC_a([0,1])$. Therefore, $\phi$ is a multiple of an isometry. 
Thus, $V:= \phi ( \CC_a([0,1])$ is an infinite-dimensional closed subspace of $\ell_\infty ([0,1], X)$.

We are going to check that each function $f_g \neq 0$ is $\tau$-continuous but $Tf_g$ is not Riemann integrable.
Since $g$ is continuous, $g(0)=g(1)=0$ and $x_n \xrightarrow{ \tau} 0$, $f_g$ is $\tau$-continuous.
Suppose $Tf_g$ is Riemann integrable. Then,
 $$ y^\ast \bigg( \int_{0}^1  Tf_g (t) dt \bigg) =\int_{0}^1 y^\ast Tf_g (t) dt= \sum \limits_{k,n}  y^\ast(Tx_n) \int_{I_k^n} g_{I_k^n} (t) dt =  0 $$ for each $y^\ast \in Y^\ast$.
 The only possible value for the Riemann integral of $Tf_g$ is $0$ due to the above equality.
Choose a partition $\mathcal{P}= \lbrace J_1, J_2, \dots, J_N \rbrace$ of $[0,1]$. Let $A= \lbrace j : 1 \leq j \leq N , \hspace{2mm} \Int{J_j} \cap K \neq \emptyset \rbrace $.
We can take $m \in \N$ such that if $j \in A$ then $J_j$ contains some interval $I_k^m$.
Hence, if $j \in A$, there is $t_j \in J_j$ such that $f_g(t_j)=\| g \| x_m $. If $j \notin A$,  then we pick any $t_j \in \Int J_j$. From the inequality $ \sum _{j \in A} \mu (J_j) \geq \mu(K) > \frac{2}{3}$, we deduce
$$ \norm[\bigg]{ \sum \limits_{j=1}^N \mu( J_j) Tf_g(t_j)  } = \norm[\bigg]{ \sum \limits_{j \in A} \mu( J_j) Tf_g(t_j) + \sum \limits_{j \notin A}  \mu( J_j) Tf_g(t_j) } \geq $$ $$ \geq \norm[\bigg]{  \sum \limits_{j \in A} \|g\| \mu( J_j) Tx_m  } - \norm[\bigg]{ \sum \limits_{j \notin A} \mu( J_j) Tf_g(t_j) } > \frac{2}{3} \|g\| - \frac{1}{3} \sup_{t \in [0,1]} \| Tf_g(t) \| =  \frac{1}{3} \| g \| .$$

Then, $Tf_g$ is Riemann integrable if and only if $g =0$ if and only if $f_g =0$. \end{proof}

The next corollary gives an affirmative answer to a question posed by Sofi in \cite{pre06123790}.

\begin{CORO}

Given an infinite-dimensional Banach space $X$, the set of all weak* continuous functions $f:[0,1] \rto X^\ast$ which are not  Riemann integrable is spaceable in $\ell_\infty ([0,1], X^\ast)$.
\end{CORO}

\begin{proof}

$X^\ast$ is not $\w ^\ast$-Schur for any infinite-dimensional Banach space $X$ due to the Josefson-Nissenzweig Theorem. Thus, the conclusion follows from Theorem \ref{THEODP}.
\end{proof}

Given a Banach space $X$, a function $f:[0,1]\rto X$ is said to be scalarly Riemann integrable if every composition $x^\ast f $ with $x^\ast \in X ^\ast$ is Riemann integrable.

We can also characterize Dunford-Pettis operators thanks to Theorem \ref{THEODP}.
The equivalence $(1) \Leftrightarrow (3)$ in the following corollary was mentioned without proof in \cite{PelczynskiRocha1980}.
\begin{CORO}
Let $X$ and $Y$ be Banach spaces and $T:X \rto Y$ be an operator. The following statements are equivalent:
\begin{enumerate}
\item $T$ is Dunford-Pettis.
\item $Tf$ is Riemann integrable for every $\w$-continuous function $f:[0,1]\rto X$.
\item $Tf$ is Riemann integrable for every scalarly Riemann integrable function $f:[0,1]\rto X$.
\end{enumerate}
\end{CORO}
\begin{proof}
$(2) \Rightarrow (1)$ is a consequence of Theorem \ref{THEODP}. Since every $\w$-continuous function $f:[0,1]\rto X$ is scalarly Riemann integrable, $(3) $ implies $(2)$. Therefore, it remains to prove $(1) \Rightarrow (3)$.
Suppose $T$ is Dunford-Pettis and fix $\Suc{ \mathcal{P}_n}{n}$ a sequence of tagged partitions of $[0,1]$ with $\| \mathcal{P}_n \| \xrightarrow{n} 0$. 
Let $f:[0,1] \rto X$ be a scalarly Riemann integrable function. Then, $x^\ast f(\mathcal{P}_n) \xrightarrow{n} \int_0^1 x^\ast f(t)dt$ for every $x^\ast \in X^\ast$. Thus, $f(\mathcal{P}_n)$ is a $\w$-Cauchy sequence in $X$, so $Tf ( \mathcal{P}_n)$ is  norm convergent to some $y \in Y$.
The limit $y$ does not depend on the sequence of tagged partitions, since if $\Suc{ \mathcal{P}'_n}{n}$ is any other sequence of tagged partitions with $\| \mathcal{P}'_n \| \xrightarrow{n} 0$, then $f(\mathcal{P}_n) - f(\mathcal{P}'_n)$ is weakly null and this in turn implies that $\| Tf(\mathcal{P}_n) - Tf(\mathcal{P}'_n) \| \xrightarrow{n} 0$.
Thus, $Tf$ is Riemann integrable.
\end{proof}

\section*{Acknowledgements}
I would like to thank Antonio Avil\'{e}s and Jos\'{e} Rodr\'{i}guez the useful suggestions and the help provided in some proofs of this paper.

\nocite{PelczynskiRocha1980}
\nocite{MR2591586}
\inputencoding{latin1}
\bibliography{WLP}
\bibliographystyle{plain}

\end{document}